\documentclass[12pt]{article}

\usepackage{graphics,amsthm, amsmath, amssymb}

\usepackage{psfig}
\usepackage{color}

\usepackage{graphics}

\usepackage{graphicx}

\expandafter\chardef\csname pre amssym.def at\endcsname=\the\catcode`\@
\catcode`\@=11

\def\undefine#1{\let#1\undefined}
\def\newsymbol#1#2#3#4#5{\let\next@\relax
 \ifnum#2=\@ne\let\next@\msafam@\else
 \ifnum#2=\tw@\let\next@\msbfam@\fi\fi
 \mathchardef#1="#3\next@#4#5}
\def\mathhexbox@#1#2#3{\relax
 \ifmmode\mathpalette{}{\m@th\mathchar"#1#2#3}%
 \else\leavevmode\hbox{$\m@th\mathchar"#1#2#3$}\fi}
\def\hexnumber@#1{\ifcase#1 0\or 1\or 2\or 3\or 4\or 5\or 6\or 7\or 8\or
 9\or A\or B\or C\or D\or E\or F\fi}

\font\tenmsa=msam10
\font\sevenmsa=msam7
\font\fivemsa=msam5
\newfam\msafam
\textfont\msafam=\tenmsa
\scriptfont\msafam=\sevenmsa
\scriptscriptfont\msafam=\fivemsa
\edef\msafam@{\hexnumber@\msafam}
\mathchardef\dabar@"0\msafam@39
\def\dashrightarrow{\mathrel{\dabar@\dabar@\mathchar"0\msafam@4B}}
\def\dashleftarrow{\mathrel{\mathchar"0\msafam@4C\dabar@\dabar@}}

\font\tenmsb=msbm10
\font\sevenmsb=msbm7
\font\fivemsb=msbm5
\newfam\msbfam
\textfont\msbfam=\tenmsb
\scriptfont\msbfam=\sevenmsb
\scriptscriptfont\msbfam=\fivemsb
\edef\msbfam@{\hexnumber@\msbfam}

\theoremstyle{plain}
\newtheorem{theorem}{Theorem}[section]
\newtheorem{lemma}[theorem]{Lemma}
\newtheorem{proposition}[theorem]{Proposition}
\newtheorem{corollary}[theorem]{Corollary}

\theoremstyle{definition}
\newtheorem{definition}[theorem]{Definition}
\newtheorem{example}[theorem]{Example}

\newtheorem{remark}[theorem]{Remark}

\def\para{\vspace{7mm}}

\def\deg{{\rm deg}}

\def\mult{{\rm mult}}

\def\cP{{\mathcal P}}

\def\dd{{\mathcal D}}
\def\cc{{\mathcal C}}

\def\N{{\mathbb N}}

\def\curva{{\sf Curve}}

\def\qed{\hfill  \framebox(5,5){}}

\def\para{\vspace{2 mm}}

\def\card{{\rm Card}}

\def\R{{\mathbb R}}
\def\C{{\mathbb C}}
\def\Z{{\mathbb Z}}

\def\K{{\mathbb K}}
\def\F{{\mathbb F}}

\def\L{{\mathbb L}}
\def\Q{{\mathbb Q}}

\def\proj{{\mathbb P}}

\def\card{{\rm card}}

\def\sup{{\rm sup}}
\def\inf{{\rm inf}}

\def\lim{{\rm lim}}

\def\ddd{{\rm d}}
\def\HH{{\rm H}}

\def\deg{{\rm deg}}

\def\mul{{\rm mult}}

\def\cP{{\mathcal P}}

\def\cc{{\mathcal C}}

\def\ii{{\mathrm i}}

\def\proj{\mathbb{P}^2(\mathbb{C})}

\begin{document}

\title{Rational Hausdorff Divisors: a New  approach to the Approximate Parametrization of Curves
\thanks{sonialuisa.rueda@upm.es, jsendra@euitt.upm.es, rafael.sendra@uah.es } }




\author{
Sonia L. Rueda  \\
Dpto. de Matem\'atica Aplicada \\
      E.T.S. Arquitectura,
 Universidad Polit\'ecnica de Madrid \\
        E-28040 Madrid, Spain 
\and Juana  Sendra \\
Dpto. Matem\'atica Aplicada  I.T. Telecomunicaci\'on. UPM, Spain\\
Research Center on Software Technologies\\ and Multimedia Systems for Sustainability (CITSEM) 
\and
 J. Rafael  Sendra\\
Dpto. de Fisica y Matem\'aticas,
        Universidad de Alcal\'a \\
      E-28871 Madrid, Spain  \\
}
\date{}          
\maketitle

\begin{abstract}
In this paper we introduce the notion of rational Hausdorff divisor, we analyze the dimension and irreducibility of its associated linear system of curves, and we prove that all irreducible real curves belonging to the linear system are rational and are at finite Hausdorff distance among them. As a consequence, we provide  a projective linear subspace where all (irreducible) elements are solutions to the approximate parametrization problem for a given algebraic plane curve. Furthermore,  we identify the linear system  with a plane curve that is shown to be rational and we develop algorithms to parametrize it analyzing its fields of parametrization. Therefore, we present a generic answer to the approximate parametrization problem. In addition, we introduce the notion of Hausdorff curve,  and we prove that every irreducible Hausdorff curve can always be parametrized  with a generic rational parametrization having coefficients depending on as many parameters as the degree of the input curve.
\end{abstract}

\noindent {\bf Keywords}: Hausdorff distance, rational Hausdorff divisor,  Hausdorff curve,  rational curve, approximate parametrization problem.

\section{Introduction}

The research in mathematics has shown that not all (mathematical) problems can be solved algorithmically, as for instance the negative solution of Hilbert's tenth problem given by Matiyasevich or the unsolvability, as a consequence of Abel's Theorem, of the  general equation of degree $n$  if $n>4$. However, even in these cases, the developed techniques to yield to these conclusions have a lot of applications in computational mathematics and in other areas of mathematics as algebra, number theory, etc.

\para

A similar, in some sense, phenomenon appears when computational mathematics mixes with applied mathematics or, more precisely, with   practical applications. In this situation, even though one may be dealing with a problem that can be solved algorithmically, and even though one has good algorithms for approaching the solution, it can happen, and  often it is the case, that the
problem has to be reformulated and analyzed from a different point of view. To be more precise in this claim, let  $\cal E$ be a mathematical entity appearing in the resolution of a practical problem (e.g. $\cal E$ is a real  polynomial) that is known, because of the nature of the treated applied problem, to satisfy  certain property ${\cal P}$ (e.g. being reducible over $\Q$) that implies the existence of certain associated objects ${\cal E}_1,\ldots,{\cal E}_n$ (e.g. the irreducible factors over $\Q$ of the polynomial), and let the goal of the problem be to compute  ${\cal E}_1,\ldots,{\cal E}_n$. However, often in practical applications,  we receive a perturbation ${\cal E}'$ of $\cal E$ instead of $\cal E$, where the property $\cal P$ does not hold anymore neither the associated objects ${\cal E}_i$ exist; for instance, a perturbation of a $\Q$-reducible polynomial will be, in general, $\Q$-irreducible and, therefore, the application of the existing polynomial factorization algorithms will  just not solve our problem. One may  try to recover the original unperturbed entity $\cal E$. Since, this is essentially impossible, a more realistic version of the problem is to determine a new  object ${\cal E}''$ near ${\cal E}'$, and satisfying
${\cal P}$, as well as  computing the associated objects ${\cal E}''_{i}$ to ${\cal E}''$. We call {\sf approximate} to an algorithm solving a problem of the above type; a solution for the illustrating example on polynomial factorization is given in \cite{GaRu}. Some papers  treating this type of problems with the same, or similar, strategy are \cite{BJ}, \cite{checos} \cite{chinos0}, \cite{HMS}, \cite{HS}, \cite{ioannis}, \cite{JutCha}, \cite{PRSS}, \cite{PSS}, \cite{PSS2}, \cite{RSS}, \cite{chinos1}; see also \cite{Robiano}.

\para

Both situations, described above,  far from being a weakness of computational mathematics, are a source of motivation and scientific challenges. Our paper is framed within the second phenomenon, and more precisely when the mathematical entity $\cal E$ is  a real algebraic plane curve, ${\cal P}$ is the property of the curve of being parametrizable by means of  rational functions (i.e. genus 0 property), and the associated objects ${\cal E}_i$ are the rational functions in one parametrization of $\cal E$. We refer to this particular problem as the {\sf approximate parametrization problem}.

\para

But, above, what does it mean  that ${\cal E}''$ is near ${\cal E}'$? Depending on the goal of the problem one uses different distances. When working with sets, as it is our case, one often uses the Hausdorff distance (see  a brief description of this distance at the end of this introduction). Indeed, the Hausdorff distance has proven to be an appropriate tool for measuring the resemblance between two geometric objects, becoming in consequence a widely used tool in fields as
computer aided design, pattern matching and pattern recognition (see for instance \cite{BYLLM}, \cite{kim}).

\para

A main difficulty when working with the Hausdorff distance is that, if not both sets are bounded, the distance between them can be infinity. Most of the papers deal with bounded real algebraic curves or with part of the curves framed into a box, and do not face the unbounded case. We, in our previous papers \cite{PRSS}, \cite{RSS}, do not restrict to the bounded case. However, in those (our) papers we provided algorithms to derive "one" solution for the approximate parametrization problem. Here we develop a theory from where, and  under the assumption that  the given curve has as many different points at infinity as degree, the set of all possible solutions of the approximate parametrization problem is described. From this analysis, one may determine an optimal or almost optimal (under certain given additional criteria) solution of the problem. For instance, one may try to provide a rational parametrization with small height (i.e. integer coefficients with small absolute value), or with a Hausdorff distance smaller than a given tolerance, or satisfying certain additional geometric features as having particular ramification points, type of singularities, tangencies, or topological graphs.

\para

Theorem 6.4, in \cite{RSS}, gives necessary conditions for having finite Hausdorff distance between two real algebraic curves.   Based on this,
we introduce the notion of Hausdorff divisor, we study the dimension and irreducibility of its linear associated system of curves (it is a projective linear subspace), and we prove that all irreducible real curves belonging to the linear system are at finite Hausdorff distance among them (see Theorem \ref{theorem-haus-divisor}). In addition, we introduce the notion of rational Hausdorff divisor, we also study the  dimension and irreducibility, and we prove that all irreducible real curves in the linear
associated system are parametrizable by means of rational functions and are at finite Hausdorff distance (see Theorem \ref{theorem-main-property-genus-divisors}). Therefore, we describe a projective linear subspace where all (irreducible) elements are solutions to the approximate parametrization problem. In a second stage, we identify the linear system of a rational Hausdorff divisor with a plane curve over the algebraic closure of a (in general) transcendental extension of $\C$. This curve is shown to be rational and we provide algorithms to parametrize it over simpler subfields (see Theorem \ref{Theorem-main-param} and its corollaries). This implies that we provide a generic answer to the approximate parametrization problem; that is a rational curve parametrization with coefficients depending polynomially on a finite set of parameters. Finally, we introduce the notion of Hausdorff curve, that essentially asks the curve to have as many different points at infinity as degree. Furthermore, we prove that every irreducible Hausdorff curve can always be parametrized  with a generic parametrization having coefficients depending on as many parameters as degree; so,  with as many degrees of freedom as the degree of the curve (see Theorem \ref{theorem-haus-curve} and its corollary). Therefore, we present an alternative algorithm to the algorithm in \cite{PRSS} that is applicable to a wider family of curves and that provides, not one, but infinitely many solutions to the problem.  We do not present any systematic study of how to proceed to choose an optimal (under a given criterion) solution from the set of infinitely many provided solutions; this is left as future research. Here instead, with an example, we illustrate the potential applicability of the method.

\para

The computations in this paper has been performed with the mathematical software Maple.

\para

\begin{center}
{\sf Hausdorff Distance}
\end{center}

\para

We briefly recall the notion of Hausdorff distance; for
further details we refer to \cite{AB}. In a metric space $(X,\ddd)$,
for $\emptyset \neq B\subset X$ and $a\in X$ we define
$\ddd(a,B)=\inf_{b\in B}\{\ddd(a,b)\}$.
Moreover,  for  $A,B\subset X\setminus \emptyset$ we define
\[ \HH_\ddd(A,B)=\max\{\sup_{a\in A}\{\ddd(a,B)\}, \sup_{b\in B}\{\ddd(b,A)\} \}.  \]
 By convention $\HH_\ddd(\emptyset,\emptyset)=0$ and, for
$\emptyset \neq A\subset X$, $\HH_\ddd(A,\emptyset)=\infty$. The
function $\HH_\ddd$ is called the {\sf Hausdorff distance induced by
$\ddd$}. In  our case, since we will be working in
$(\R^2,\ddd)$,  $\ddd$ being the usual  Euclidean
distance, we simplify the notation writing $\HH(A,B)$.

\section{Hausdorff  Divisors}\label{sec-haus-div}

Throughout this paper we denote by $\proj$ the projective plane over the field $\C$ of complex numbers.  Let us start recalling the notion of divisor. A divisor is, intuitively speaking, a way of describing finite collections of points in $\proj$ with assigned (maybe negative) multiplicities. More precisely, a  {\sf divisor} in $\proj$ is a  formal expression
\[ \sum_{i=1}^{m} s_i P_i \]
where $s_i\in \Z$  and $P_i$ are different points in $\proj$; if $s_i$ are all non-negative integers, the divisor is called {\sf effective}. In this paper, we are only interested in effective divisors; non-effective divisors are used when poles of rational functions need to be analyzed. We define the {\sf degree} of the divisor $D=\sum_{i=1}^{m} s_i P_i $ as the number $$\deg(D):=\sum_{i=1}^{m} s_i.$$

\para

\begin{definition}\label{def-divisor-hau}
We say that a divisor $\sum_{i=1}^{m} s_i P_i$ is a {\sf Hausdorff divisor} if, for all $i\in\{1,\ldots,m\}$, $s_i=1$  and $P_i$ is of the form
$(a:b:0)\in \proj$. \hfill$\bullet$
\end{definition}

\para

Let $n$ be a positive integer, and let $\cc$ be a projective algebraic plane curve of degree $n$. $\cc$ will be defined by a homogeneous polynomial $F(x,y,z)$ with coefficients in $\C$, then we can identify $\cc$ with the projective point given by its coefficients after fixing a term order. For instance, if $n=2$, let us fix e.g. the order $y^2<xy<x^2<yz<xz<z^2$, then the circle $x^2+y^2-z^2=0$ is seen as $(1:0:1:0:0:-1)\in {\mathbb P}^5(\C)$ and any other conic (including double lines) is a point in ${\mathbb P}^5(\C)$. In this situation, the set of all projective curves of degree $n$ is identified with the projective space ${\mathbb P}^{\ell-1}(\C)$, where $\ell=(n+1)(n+2)/2$. Now, a {\sf  linear system of curves of degree $n$} is a linear subspace of ${\mathbb P}^{\ell}(\C)$. Considering parametric equations of the linear system one can see the system as a homogeneous form whose coefficients are those equations. This form is called the {\sf defining polynomial of the linear system}.
For instance, let $n=1$ and $x<y<z$, then the linear subspace,  of degree 1, parametrized as $(\lambda:\mu:0)$ corresponds to the form $\lambda x +\mu y$ that defines the pencil of projective lines passing through the origin $(0:0:1)$. We, indistinctly, will see the linear system as a projective linear subspace or as a homogeneous polynomial.

 \para

Associated with an effective divisor, one can consider a linear system of curves for a positive integer $n$, big enough. More precisely, let $D=\sum_{i=1}^{m} s_i P_i$  be an effective divisor, then we  consider the set of all projective curves of degree $n$ passing through $P_i$ with multiplicity, at least, $s_i$, for $i=1,\ldots, m$. Observe that these requirements are linear conditions on the coefficients of the generic form of degree $n$. Therefore, this set is a linear system of curves. We denote it by ${\cal H}(n,D)$.

\para

\begin{definition}\label{def-hau-linear-system}
Let $D$ be an $n$-degree Hausdorff divisor, the linear system ${\cal H}(n,D)$ is called the {\sf Hausdorff linear system associated to $D$.} \hfill$\bullet$
\end{definition}

\para

A natural question is  the analysis of the dimension of a linear system. In general, if $D$ is an effective divisor it holds (see Theorem 2.59 in \cite{SWP})
 \begin{equation}\label{equation-dim}
 \dim({\cal H}(n,D))\geq \frac{n(n+3)}{2}-\sum_{i=1}^{m} \frac{s_i(s_i+1)}{2}.
\end{equation}
One may also consider the notion of divisor in  general position (see Section 2.4 in \cite{SWP}). From (\ref{equation-dim}), the following result is deduced.

\para

\begin{proposition}\label{prop-dim}
Let $D$ be an $n$-degree Hausdorff divisor. Then $$\dim({\cal H}(n,D)\geq \frac{n(n+1)}{2}.$$
\end{proposition}

\para

Let $D$ be an  effective divisor such that ${\mathcal H}(n,D)\neq \emptyset$, and let $H(\Lambda,x,y,z)$ be its defining polynomial, where
$\Lambda$ is a tuple of parameters. For each specialization  $\Lambda_0$ of $\Lambda$, taking values in $\C$, we get a projective curve, namely the curve defined by  $H(\Lambda_0,x,y,z)$. Alternatively, we can see ${\cal H}(n,D)$ as a projective plane curve over  the algebraic closure of $\C(\Lambda)$ defined by $H(\Lambda,x,y,z)$. This, motivates the following definition.

\para

\begin{definition}\label{def-curve-asoc-linear-sytema}
Let $D$ be an effective divisor, let $\emptyset \neq \overline{\cal H}\subseteq {\cal H}(n,D)$, and let $\overline{H}(\overline{\Lambda},x,y,z)$ be the defining polynomial of $\overline{\cal H}$. The projective plane curve defined by $\overline{H}(\overline{\Lambda},x,y,z)$, over the algebraic closure of $\C(\overline{\Lambda})$, is called the {\sf projective algebraic curve associated to $\overline{\cal H}$} and we denote it by $\curva(\overline{\cal H})$; in general we will identify $\overline{\cal H}$ and $\curva(\overline{\cal H})$.\hfill$\bullet$
\end{definition}

\para

\begin{example}\label{example-curve}
We illustrate the previous concept in this example. Let $D=(0:0:1).$
The defining polynomial of ${\cal H}(1,D)$ is (here $\Lambda=(\lambda_{1},\lambda_{2})$)
\[ H(\Lambda,x,y,z)=\lambda_1 x+\lambda_2 y \]
Therefore, ${\cal H}(1,D)$ consists in all lines in $\mathbb{P}^2(\C)$ passing through $(0:0:1)$. However, $\curva({\cal H}(1,D)$ is a particular line in ${\mathbb P}^{2}(\overline{\C(\Lambda)})$ namely the line $y=-\lambda_1/\lambda_2 x$, where $\overline{\C(\Lambda)}$ is the algebraic closure of $\C(\Lambda)$. \hfill $\Box$
\end{example}



\para

Our main goal will be to parametrize the projective algebraic curve associated to a linear system. This implies that the curve has to be irreducible, and hence the next notion appears naturally.

\para

\begin{definition}\label{Def-irred}
Let $D$ be an  effective divisor such that ${\mathcal H}(n,D)\neq \emptyset$. We say that $D$ is {\sf irreducible} if  $\curva({\mathcal H}(n,D)) $ is irreducible; that is, if the defining polynomial  $H(\Lambda,x,y,z)$  of the linear system ${\mathcal H}(n,D)$ is  irreducible over the algebraic closure of $\C(\Lambda)$. \hfill$\bullet$
\end{definition}

\para

\begin{example}\label{example-irreducible}
  Let $D=2(1:0:0).$ The defining polynomial of ${\cal H}(2,D)$ is (here $\Lambda=(\lambda_{1},\lambda_{2},\lambda_3)$)
\[ H(\Lambda,x,y,z)=\lambda_{{1}}{z}^{2}+\lambda_{{2}}yz+\lambda_{{3}}{y}^{2}. \]
This polynomial is irreducible over $\C$. However, over  the algebraic closure $\overline{\C(\Lambda)}$    of $\C(\Lambda)$, $H$ factors as
\[  H(\Lambda,x,y,z)= \lambda_{{1}} \left( z+\frac{1}{2}\,{\frac {\lambda_{{2}}y}{\lambda_{{1}}}}
 \right) ^{2}+ \left( \lambda_{{3}}-\frac{1}{4}\,{\frac {{\lambda_{{2}}}^{2}}{
\lambda_{{1}}}} \right) {y}^{2}
=\] \[ =\left( \sqrt {\lambda_{{1}}} \left( z+\frac{1}{2}\,{\frac {\lambda_{{2}}y}{
\lambda_{{1}}}} \right) +\frac{1}{2}\,\ii\sqrt {4\,\lambda_{{3}}-{\frac {{
\lambda_{{2}}}^{2}}{\lambda_{{1}}}}}y \right)  \cdot \left( \sqrt {\lambda_{
{1}}} \left( z+\frac{1}{2}\,{\frac {\lambda_{{2}}y}{\lambda_{{1}}}} \right)
-\frac{1}{2}\,\ii\sqrt {4\,\lambda_{{3}}-{\frac {{\lambda_{{2}}}^{2}}{\lambda_{{1}
}}}}y \right).\]
Thus $D$ is reducible or, equivalently $\curva({\cal H}(2,D))$ is reducible; indeed, it is a pair of lines.
Note that $D$ is not Hausdorff. \hfill $\Box$
\end{example}

\para

In the next theorem, we study the irreducibility of $\curva({\mathcal H}(n,D))$, when $D$ is a Hausdorff divisor; observe that a Hausdorff linear system is never empty  (see Prop. \ref{prop-dim}), and hence $\curva({\mathcal H}(n,D))$ always exists. We start with the following lemma.

\para

\begin{lemma}\label{lemma-irred}
Let $D$ be an $n$-degree Hausdorff divisor. The defining polynomial of ${\cal H}(n,D)$ is irreducible over $\C$.
\end{lemma}
\begin{proof} Let $H(\Lambda,x,y,z)$ be the defining polynomial of ${\cal H}(n,D)$ and let $\F$ be the algebraic closure of $\C(\Lambda)$. If $H(\Lambda,x,y,z)$ factors over $\C$, with factors depending not  only on $\Lambda$, then all curves in the linear system are reducible. So, to prove the statement, we find an specific irreducible projective curve in ${\cal H}(n,D)$. Let us assume that $D=\sum_{i=1}^{n} (a_i:b_i:0)$. Let $(a:b:0)$ be different to all points in $D$. We consider the projective curve $\cc$ defined by
\[ F(x,y,z)= z (bx-ay)^{n-1}-\prod_{i=1}^{n} (b_i x-a_i y). \]
Since $(a:b:0)$ is different to $(a_i:b_i:0)$, $F$ is irreducible and clearly $\cc\in {\mathcal H}(n,D)$.
\end{proof}

\para

\begin{theorem}\label{theorem-irred}
Let $D$ be an $n$-degree Hausdorff divisor. Then, $D$ is irreducible.
\end{theorem}
\begin{proof} Let $H(\Lambda,x,y,z)$, ${\cal H}(n,D)$ and $\F$ as in the proof of Lemma \ref{lemma-irred}. We may assume w.l.o.g. that $H$ is monic w.r.t. $y$ (this is equivalent to $(0:1:0)\not\in D$): indeed, if it is not the case, we can always perform a projective change of coordinates over $\C$, such that $(0:1:0)\not\in D$ and $D$ stays  Hausdorff; then  the irreducibility of $H$ over $\F$ is preserved. Since $D$ is Hausdorff, by Lemma \ref{lemma-irred},  $x$ does not divide $H$. Therefore, $H$ is irreducible over $\F$ iff $h(y,z)=H(1,y,z)$ is irreducible over $\F$.  $h(y,z)$ is monic in $y$. Moreover, since $D$ is Hausdorff, $h(y,0)$ is square-free. Therefore, using Theorem 5.5.2 in \cite{Winkler}, we have that $h$ is irreducible over $\F$ iff $h$ is irreducible over $\C$. Now, the result follows from Lemma \ref{lemma-irred}.
\end{proof}

\para

\begin{corollary}\label{corollary-irred}
Let $D$ be an $n$-degree Hausdorff divisor, and let $\overline{H}(\overline{\Lambda},x,y,z)$ be the defining polynomial of a non-empty  linear subsystem $\overline{\cal H}\subseteq{\cal H}(n,D)$. Then, $\overline{H}$ is irreducible over the algebraic closure of $\C(\overline{\Lambda})$ if and only if $\overline{H}$ is irreducible over $\C$; that is, $\curva(\overline{\cal H})$ is irreducible if and only if $\overline{H}$ is irreducible over $\C$.
\end{corollary}
\begin{proof}
Let left-right implication is trivial. The right-left implication follows as the proof of Theorem \ref{theorem-irred} but using the irreducibility of $\overline{H}$ over $\C$, instead of Lemma \ref{lemma-irred}.
\end{proof}

\para

The next theorem states the main result on Hausdorff divisors. For this purpose, if $\cc$ is the projective algebraic curve defined by the form $F(x,y,z)$, and it is different to the line at infinity $z=0$, we denote by $\cc_a$ the affine algebraic curve defined by $F(x,y,1)$. Furthermore, for an affine algebraic curve $\cc_{a}$ we denote by $\cc_{a}^{\infty}$ the points at infinity of $\cc_a$. We recall that an affine curve is {\sf real} if it contains infinitely many real points.

\para

\begin{theorem}\label{theorem-haus-divisor}
Let $D$ be an $n$-degree Hausdorff divisor.  For every two real irreducible curves  $\cc_1,\cc_2\in {\cal H}(n,D)$, such that $\deg(\cc_{i,a})=n$, it holds that
$$\HH(\cc_{1,a}\cap \R^2, \cc_{2,a}\cap \R^2)<\infty.$$
\end{theorem}
\begin{proof} Let $D=\sum_{i=1}^{n} P_i$. Since $\cc_1,\cc_2\in {\cal H}(n,D)$ then  $\cc_{1,a}^{\infty}=\cc_{2,a}^{\infty}=\{P_1,\ldots,P_n\}$. Moreover, $\card(\cc_{1,a}^{\infty})=\card(\cc_{2,a}^{\infty})=\deg(\cc_{1,a})=\deg(\cc_{2,a})$. Now, the result follows from Theorem 6.4. in \cite{RSS}
\end{proof}

\para

In the following we find necessary conditions on the Hausdorff divisor  $D=\sum_{i=1}^{n} P_i$ such that ${\cal H}(n,D)$ contains curves  verifying the hypotheses  of Theorem \ref{theorem-haus-divisor}. For this purpose, we will use the concept of family of conjugate points that can be introduced as follows; see Def. 3.15 in \cite{SWP} for further details. Let $\K$ be a subfield of $\C$, e.g. $\K=\R$, then a finite family of points is ${\mathbb P}^2(\C)$ is {\sf $\K$-conjugate} if it can be expressed as
\[ \{(p_1(t):p_2(t):p_3(t))\,|\, m(t)=0\} \]
where $p_i,m\in \K[t]$ and $gcd(p_1,p_2,p_3)=1$; for instance, the points in ${\cal F}:=\{(\pm \ii:1:0)\}$ are $\Q$-conjugated since ${\cal F}=\{(t:1:0)\,|\,t^2+1=0\}$.

\para

Let $\cc \in  {\cal H}(n,D)$ be  such that $\deg(\cc_{a})=n$, and $\cc_{a}$ is real and irreducible. Let $F(x,y,z)$ be the defining polynomial of $\cc$. Then, $\{P_1,\ldots,P_n\}$ is the family of conjugate points
$\{(t:h:0)| F(t,h,0)=0\}$. Moreover, since $\cc$ is real, then $F$ is a real polynomial (see Lemma 7.2 in \cite{SWP}), and thus the family is $\R$-conjugated. This motivates the following definition.

\para

\begin{definition}\label{def-real-haus-divisor}
Let $\K$ be a subfield of $\C$. We say that a Hausdorff divisor $D=\sum_{i=1}^{m} P_i$ is {\sf $\K$-definable} if $\{P_1,\ldots,P_m\}$ is a $\K$-conjugate family of points. \hfill$\bullet$
\end{definition}

\para

In the next examples, we illustrate the notion of $\R$-definability as well as Theorem \ref{theorem-haus-divisor}.

\para

\begin{example}\label{ex-k-def-haus}
We consider the Hausdorff divisor ($\ii$ is the imaginary unit)
\[ D=\left(\frac{1}{2} \sqrt {2}+\frac{1}{2} \ii\sqrt {2}+1:\ii:0\right)+\left(-\frac{1}{2} \sqrt {2}+\frac{1}{2} \ii\sqrt {
2}+1:-\ii:0\right)+\] \[ +\left(-\frac{1}{2} \sqrt {2}-\frac{1}{2} \ii\sqrt {2}+1:\ii:0\right)+\left(\frac{1}{2} \sqrt {2}-\frac{1}{2}
 \ii\sqrt {2}+1:-\ii:0\right).
 \]
$D$ can be expressed as $D= \sum  (\alpha+1:\alpha^2:0)$, where $\alpha^4+1=0$.
Therefore, $D$ is $\R$-definable. \hfill $\Box$
\end{example}

\para

\begin{example}\label{example-cuadrics}
We consider the $4$-degree Hausdorff divisor $$D=(1:1:0)+(-1:1:0)+(0:1:0)+(1:0:0).$$
The defining polynomial of ${\cal H}(4,D)$ is

\para

\noindent $H(x,y,z)=\lambda_{11}{z}^{4}+\lambda_{10}y{z}^{3}+\lambda_{9}{y}^{2}{z}^{2}+\lambda_{8}{y}^{3}z+
\lambda_{7}x{z}^{3}+\lambda_{6}xy{z}^{2}+\lambda_{5}x{y}^{2}z-\lambda_{1}x{y}^{3}+\lambda_{4}{x}^{2}{z}^{2}+\lambda_{3}{x}^{2}yz+\lambda_{2}{x}^{3}
z+\lambda_{1}{x}^{3}y.$

\para

Observe that the number of parameters $\lambda_i$ is 11, and hence $\dim({\cal H}(4,D))=10$; compare to Prop. \ref{prop-dim}. In Fig. \ref{figure-1} one may see 8 different curves in the linear system. Observe that all of them have asymptotes in the direction of the vectors $(1,1),(-1,1),(1,0),(0,1)$. \hfill $\Box$
\begin{figure}
  \centering
\includegraphics[width=6cm]{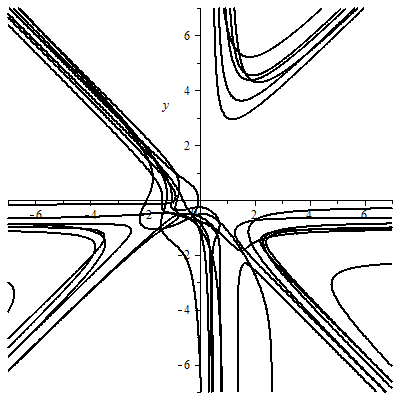}
\caption{Some curves in ${\cal H}(4,D)$ in Example \ref{example-cuadrics}}
\label{figure-1}
\end{figure}
\end{example}

\para

\begin{center}
{\sf Conics: 2-degree $\R$-definable Hausdorff divisors}
\end{center}

\para

 In this subsection we analyze the 2-degree $\R$-definable Hausdorff divisors. We distinguish two cases: first the two points of the divisor are real, and second the two points are complex in which case they have to be conjugated because of the $\R$-definability.

\vspace*{1 mm}

\noindent {\sf [Real points: the non-compact case]} We consider  2-degree Hausdorff divisors with real points. We distinguish several cases. We start with $D=(1:0:0)+(0:1:0)$.  The defining polynomial of ${\mathcal H}(2,D)$ is
   $
    H=a_{{0,0}}{z}^{2}+a_{{0,1}}yz+a_{{1,0}}xz+a_{{1,1}}xy.
$
We may assume w.l.o.g. that $a_{1,1}\neq 0$, since otherwise for all $\cc$, in the linear system, $\deg(\cc_a)=1$. Then, $H(x,y,1)$ can be expressed as
\[H(x,y,1)=\left( x+{\frac {a_{{0,1}}}{a_{{1,1}}}} \right)  \left( y+{\frac {a_{
{1,0}}}{a_{{1,1}}}} \right) +{\frac {a_{{0,0}}}{a_{{1,1}}}}-{\frac {a_
{{0,1}}a_{{1,0}}}{{a_{{1,1}}}^{2}}}.
\]
Now, observe that all real irreducible affine curves derived from the system are hyperbolas with parallel asymptotes, indeed with direction vectors $(1,0)$ and $(0,1)$ (compare to Theorem 3 in  \cite{Co} or Lemma 6.1 in \cite{PRSS}), and hence Theorem \ref{theorem-haus-divisor} holds.
\\
Second, we take $D=(a:1:0)+(b:1:0)$, with  $a,b\in \R$, $a\neq b$. The defining polynomial of ${\mathcal H}(2,D)$ is
$H=a_{{0,0}}{z}^{2}+a_{{0,1}}yz+ba_{{2,0}}a{y}^{2}+a_{{1,0}}xz-a_{{2,0}}a
xy-xya_{{2,0}}b+a_{{2,0}}{x}^{2}.$
We may assume w.l.o.g. that $a_{2,0}\neq 0$, since otherwise for all $\cc$, in the linear system, $\deg(\cc_a)=1$. Then, $H(x,y,1)$ can be expressed as \[ \frac{1}{4} \left( 2x+{\frac {a_{{1,0}}}{a_{{2,0}}}}-ay-by \right) ^{2}-\frac{1}{4} \left(  \left( a-b \right) y-{\frac {\Delta}{ \left( a-b \right) {
a_{{2,0}}}^{2}}} \right) ^{2}+
\]
\[+ \frac{1}{4}{\frac {4a_{{0,0}}a_{{2,0}}-{a_{{
1,0}}}^{2}}{{a_{{2,0}}}^{2}}} +\frac{1}{8}{\frac {{\Delta}^{2}}{ \left( a-b
 \right) ^{2}{a_{{2,0}}}^{4}}},  \]
where $\Delta=a_{{2,0}} \left( 2a_{{0,1}}+a_{{1,0}}a+a_{{1,0}}b \right)$.
Now, observe that all real irreducible affine curves derived from the system are hyperbolas with parallel asymptotes, indeed with direction vectors $(a,1)$ and $(b,1)$ (compare to Theorem 3 in  \cite{Co} or Lemma 6.1 in \cite{PRSS}), and hence Theorem \ref{theorem-haus-divisor} holds.
\\
Third, we take $D=(1:0:0)+(b:1:0)$, with  $b\in \R$, $b\neq 0$. The defining polynomial of ${\mathcal H}(2,D)$ is
$H=a_{{0,0}}{z}^{2}+a_{{0,1}}yz-a_{{1,1}}b{y}^{2}+a_{{1,0}}xz+a_{{1,1}}xy.$
We may assume w.l.o.g. that $a_{1,1}\neq 0$, since otherwise for all $\cc$, in the linear system, $\deg(\cc_a)=1$. Then, $H(x,y,1)$ can be expressed as
 \[  \left( y-{\frac {a_{{0,1}}}{2a_{{1,1}}b}}-{\frac {x}{2b}}
 \right) ^{2}-\frac{1}{4b^2} \left( x+2{\frac {a_{{1,0}}b}{a_{{1,1}}}
}+{\frac {a_{{0,1}}}{a_{{1,1}}}} \right) ^{2}-\frac{1}{b^2}\left({\frac {ba_{{0,0}}}{a
_{{1,1}}}}+{\frac {{a_{{0,1}}}^{2}}{{a_{{1,1}}}^{2}}}\right)\]
\[+\frac{1}{4b^2} \left( 2
{\frac {a_{{1,0}}b}{a_{{1,1}}}}+{\frac {a_{{0,1}}}{a_{{1,1}}}}
 \right) ^{2} .  \]
Now, observe that all real irreducible affine curves derived from the system are hyperbolas with parallel asymptotes, indeed with direction vectors $(1,0)$ and $(b,1)$ (compare to Theorem 3 in  \cite{Co} or Lemma 6.1 in \cite{PRSS}), and hence Theorem \ref{theorem-haus-divisor} holds.

\vspace*{1 mm}

\noindent {\sf [Complex points: the compact case]} Since both points have to be complex and conjugated, we can assume w.l.o.g. that $D$ is of the form
$D=(a+\ii:1:0)+(a-\ii:1:0)$, where $\ii$ is the imaginary unit. The defining polynomial of ${\mathcal H}(2,D)$ is
$H=a_{{0,0}}{z}^{2}+a_{{0,1}}yz+{y}^{2}a_{{2,0}}{a}^{2}+{y}^{2}a_{{2,0}}+
a_{{1,0}}xz-2a_{{2,0}}axy+a_{{2,0}}{x}^{2}
$.
We may assume w.l.o.g. that $a_{2,0}\neq 0$, since otherwise for all $\cc$, in the linear system, $\deg(\cc_a)=1$. Then, $H(x,y,1)$ can be expressed as
\[  \left( x+{\frac {a_{{1,0}}}{2a_{{2,0}}}}-ay \right) ^{2}+ \left(
y+{\frac {a_{{0,1}} +a_{{1,0}}a   }{2a_{{2,0}}}} \right) ^{2}-{\frac {-4a_{{0,0}}a_{{2,0}}+{a_{{1,0}}}^{2}+
{a_{{0,1}}}^{2}+2a_{{0,1}}a_{{1,0}}a+{a_{{1,0}}}^{2}{a}^{2}}{{4a_{{2,0
}}}^{2}}}.
\]
So, if $a\neq 0$ we get ellipses and for $a=0$ we get  circles (note that for $a=0$ the divisor is defined by the cyclic points)
\[  \left( x+
{\frac {a_{{1,0}}}{2a_{{2,0}}}} \right) ^{2}+ \left( y+{\frac {a_{{0,1}}}{2a_{{2,0}}}} \right) ^{2}-{\frac {-4a_{{0
,0}}a_{{2,0}}+{a_{{1,0}}}^{2}+{a_{{0,1}}}^{2}}{{4a_{{2,0}}}^{2}}}.
\]
In both cases, the statement is Theorem \ref{theorem-haus-divisor} clearly holds.

\para

\section{Rational Hausdorff  Divisors}\label{sec-rat-haus-div}

We start this section recalling briefly the concept of rational curve. An algebraic curve is called {\sf rational} if it can be parametrized by means of rational functions; in other words, if $F(x,y,z)$ is the homogeneous polynomial defining a projective curve $\cc$, then  $\cc$ is rational if there exist three polynomials $p_1(t),p_2(t),p_3(t)$, not all constant, such that $\gcd(p_1,p_2,p_3)=1$, and $F(p_1(t),p_2(t),p_3(t))$. In this case, $(p_1(t),p_2(t),p_3(t))$ is a {\sf rational parametrization} of $\cc$. If $\cc$ is not the line at infinity $z=0$, we usually  write the parametrization as $(p_1(t)/p_3(t),p_2(t)/p_3(t),1)$. The rationality of a curve can be deduced from its genus. An irreducible curve is rational if and only if its genus is 0. The {\sf genus}, intuitively speaking, measures the difference between the maximum of singularities the curve may have an the actual number of them. More precisely, the genus is given  by the formula
\[ \frac{(\deg(\cc)-1)(\deg(\cc)-2)}{2}-\sum_{P\in \cc} \frac{\mult(\cc,P)(\mul(\cc,P)-1)}{2} \]
where $\deg(\cc)$ denotes the degree of $\cc$ (i.e. the degree of the form $F$), $\mult(\cc,P)$ denotes the multiplicity of $\cc$ at $P$, and where the sum is taken also over the infinitely near, or neighboring,  points (see Chapter 3 in \cite{SWP} for further details). Note that if $\cc$ is irreducible and has a point of multiplicity $(\deg(\cc)-1)$ then the genus is 0, and hence $\cc$ is rational. Curves satisfying this particular case are called {\sf monomial curves}.

\para

In this section, we introduce the notion of rational divisor or genus 0 divisor. Similarly, one can consider the concept of genus $g$ divisor but, here, we are only interested in the genus 0 case.  The definition we give here focuses on singularities of ordinary type; i.e. all tangents at the point are different. The case of non-ordinary singularities can also be introduced. For that case, associated to each point in the divisor a sequence of linear transformations and "neighboring" divisors would have to be attached.

\para

\begin{definition}\label{def-rational-div}
Let $n\in \N$, $n>0$, and $D=\sum_{i=1}^{m} s_i P_i$ an effective divisor. If $n\in \{1,2\}$, we say that $D$ is an {\sf $n$-rational divisor} if $\deg(D)=1$. If $n>2$, we say that $D$ is an {\sf $n$-rational divisor} if $s_i>1$ for $i=1,\ldots,m$, and
\[(n-1)(n-2)=\sum_{i=1}^{m}  s_i(s_i-1). \]
If  $D$ is $n$-rational, and only contains a point,  we say that $D$ is an {\sf $n$-monomial divisor}. \hfill$\bullet$
\end{definition}

\para




 Note that $D=P$ is a $1$-monomial and a $2$-monomial divisor. In general, for $n>2$, $D=(n-1)P$ is an $n$-monomial divisor. On the other hand, for $n=3$ the only possible rational divisors are monomial, i.e. $D=2P$, while for $n>3$ the situation is open to more possibilities; for instance, for $n=4$, one has $D=3P$ or $D=2P_1+2P_2+2P_3$.

\para

The singular locus, and hence the rational divisor, of a real irreducible plane curve can be decomposed as the union of conjugate singularities (see Section 3.3 in \cite{SWP}; more particularly Corollary 3.23). We introduce the next definition.

\para

\begin{definition}\label{def-real-haus-divisor}
Let $\K$ be a subfield of $\C$. We say that a rational divisor $D$ is {\sf $\K$-definable} if $D$ can be expressed as
\[ D= \sum_{i=1}^{m_1} s_{1} P_{1,i} +\cdots + \sum_{i=1}^{m_k} s_{k} P_{k,i} \]
where $\{P_{j,1},\ldots,P_{j,m_j}\}$ is a family of $\K$-conjugated points, for $j=1,\ldots,k$. \hfill$\bullet$
\end{definition}

\para

Observe that a monomial divisor is $\R$-definable if and only if the point in the divisor is real.
The next results deal with the dimension.

\para

\begin{theorem}\label{theorem-dim-rational-div}
Let $D$ be an $n$-rational divisor, then $\dim({\cal H}(n,D))\geq 3n-1-\deg(D)$.
\end{theorem}
\begin{proof}
It follows from inequality (\ref{equation-dim}) and Def. \ref{def-rational-div}.
\end{proof}

\para

\begin{corollary}
Let $D$ be an $n$-monomial divisor, then $\dim({\cal H}(n,D))\geq 2n$.
\end{corollary}

\para

The main property  on this type of divisors is the following.
\para

\begin{theorem}\label{theorem-main-property-genus-divisors}
Let $D$ be an $n$-rational divisor.
\begin{enumerate}
\item Every irreducible curve in ${\cal H}(n,D)$ is rational.
\item If $D$ is irreducible (see Def. \ref{Def-irred}), then $\curva({\cal H}(n,D))$ is rational.
\end{enumerate}
\end{theorem}
\begin{proof}
Since $D$ is rational, if the curve is irreducible, its genus is zero. So the curve is rational.
\end{proof}

\para

Our next step is to combine both notions, Hausdorff and rational divisor.

\para

\begin{definition}\label{def-real-haus-divisor}
We say that an effective divisor $D$ is an {\sf $n$-rational Hausdorff divisor} if
$D$ can be expressed as
\[ D= D_1+D_2 \]
where $D_1$ is  $n$-degree Hausdorff, and $D_2$ is $n$-rational and no point in $D_2$ is on the line $z=0$ (i.e. all points in $D_2$ are affine). If both $D_1,D_2$ are $\K$-definable, we say that $D$ is {\sf $\K$-definable}, where $\K$ is a subfield of $\C$.
Given a rational Hausdorff divisor $D$, we denote by $D_H$ and by $D_S$ the Hausdorff and the singular part of $D$, respectively. In addition, we say that ${\cal H}(n,D)$ is the {\sf rational Hausdorff linear space} associated to $D$.
\hfill$\bullet$
\end{definition}

\para

Note that, since all points in $D_1$ have to be at infinity and all points in $D_2$ have to be affine, the decomposition $D_1+D_2$ is unique.

\para

Now, we analyze ${\cal H}(n,D)$ where $D$ is an $n$-rational Hausdorff divisor. First, we observe that every irreducible curve in ${\cal H}(n,D)$ is smooth at the line $z=0$ and rational. Let us study the dimension. By  Proposition \ref{prop-dim} and Theorem \ref{theorem-dim-rational-div}, we get the following result.

\para

\begin{theorem}\label{theorem-dim-rational-hauss-divisor}
Let $D=D_H+D_S$ be an $n$-rational Hausdorff divisor then
\[ \dim({\cal H}(n,D))\geq 2n-1-\deg(D_S). \]
\end{theorem}

\para

\begin{corollary}\label{col-dim-monomial-rat-haus-div}
If $D$ is an $n$-monomial Hausdorff divisor, then $\dim({\cal H}(n,D))\geq n.$
\end{corollary}

\para

We illustrate the previous results by some examples.

\para

\begin{example}\label{example-cuadrics-2}
We consider the divisor
$D=(1:1:0)+(-1:1:0)+(0:1:0)+(1:0:0)+2(3:-2:1)+2(1:1:1)+2(2:3:1).$
$D$ can be expressed as $D=D_H+D_S$ where
\[ \begin{array}{l}
D_H=(1:1:0)+(-1:1:0)+(0:1:0)+(1:0:0),\\
D_S=2(3:-2:1)+2(1:1:1)+2(2:3:1).
 \end{array} \]
Note that $D_H$ is a $4$-degree Hausdorff divisor (indeed, the one in Example \ref{example-cuadrics}) and $D_S$ is a $4$-rational divisor. So, $D$ is a $4$-rational Hausdorff divisor, in fact $\R$-definable.
The defining polynomial of ${\cal H}(4,D)$ is (where $\Lambda=(\lambda_1,\lambda_2)$)

\[
H(\Lambda,x,y,z)=( {\frac {65}{2}}\lambda_{2}-{\frac {8175}{98}}\lambda_{1}
 ) {z}^{4}+ ( 17\lambda_{2}-{\frac {1518}{49}}\lambda_{1}
 ) y{z}^{3}+ ( -{\frac {29}{2}}\lambda_{2}+{\frac {2787}{98
}}\lambda_{1} ) {y}^{2}{z}^{2}\]\[+ \lambda_{2}{y}^{3}z+ ( -97\lambda_{2}+{\frac {11618}{49}}\lambda_{1} ) x{z}^{3}+
( \frac{11}{2}\lambda_{2}-{\frac {1789}{98}}\lambda_{1} ) xy{z}^{2} + ( \frac{9}{2}\lambda_{2}-{\frac {121}{14}}\lambda_{1} ) x{y}^{2}z\]\[-\lambda_{1}x{
y}^{3}+ ( {\frac {143}{2}}\lambda_{2}-{\frac {16873}{98}}\lambda_{1}) {x}^{2}{z}^{2}+ ( -\frac{11}{2}\lambda_{2}+{\frac {163}{14}}
\lambda_{1} ) {x}^{2}yz+ ( -15\lambda_{2}+{\frac {254}{7}}
\lambda_{1} ) {x}^{3}z+\lambda_{1}{x}^{3}y.
 \]
\noindent Observe that the number of parameters $\lambda_{i}$ is 2, and hence $\dim({\cal H}(4,D))=1$; check with Theorem \ref{theorem-dim-rational-hauss-divisor}. In Fig. \ref{figure-2} one may see 5 different curves in the linear system.  \hfill $\Box$
\begin{figure}[h]
  \centering
\includegraphics[width=5cm]{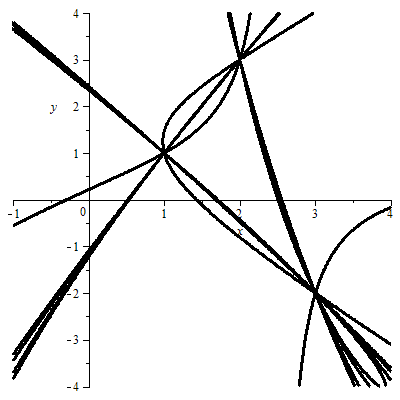} \hspace*{2mm} \includegraphics[width=5cm]{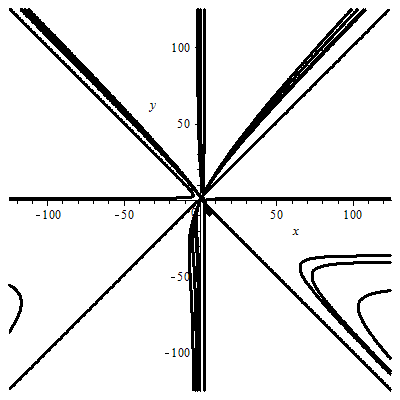}
\caption{Some curves in ${\cal H}(4,D)$ of Example \ref{example-cuadrics-2}. Left: Zoom at the singular area. Right: general view}
\label{figure-2}
\end{figure}
\end{example}

\para

\begin{example}\label{example-cuadrics-3}
In Example \ref{example-cuadrics-2}, we took
$D=D_H+D_S$ with
\[ \begin{array}{l}
D_H=(1:1:0)+(-1:1:0)+(0:1:0)+(1:0:0),\\
D_S=2(3:-2:1)+2(1:1:1)+2(2:3:1),
\end{array}\]
and ${\cal H}(4,D)$ was irreducible over $\C$. However, if   $D_S=2(2:2:1)+2(1:1:1)+2(2:3:1)$, ${\cal H}(4,D)$ decomposes as the union of two lines and a system of conics. More precisely, the defining polynomial is
$$4\left( x-2z \right)  \left( x-y \right)  \left( 2x\lambda_{2}z
+2xy\lambda_{1}-{z}^{2}\lambda_{2}+9{z}^{2}\lambda_{1}-yz\lambda_{2}-13y
z\lambda_{1}+2\lambda_{1}{y}^{2} \right)
.$$
Obviously the reason, in this example, is that  two double points, namely $(2:2:1),(1:1:1)$, and one simple point, namely $(1:1:0)$, are on the same line. \hfill$\Box$
\end{example}

\para

The next theorem shows how to analyze the irreducibility of rational Hausdorff divisors.

\para

\begin{theorem}\label{theorem-irred-2}
Let $D$ be an $n$-rational Hausdorff divisor. Then, $D$ is irreducible (see Def. \ref{Def-irred}) if and only if the defining polynomial of ${\cal H}(n,D)$ is irreducible over $\C$.
\end{theorem}
\begin{proof}
Let $D=D_H+D_S$. Now, observe that ${\cal H}(n,D)$ is a linear subsystem of ${\cal H}(n,D_H)$. Now the result follows from Corollary \ref{corollary-irred}.
\end{proof}

\para

The bounds in Theorem \ref{theorem-dim-rational-hauss-divisor} and Corollary \ref{col-dim-monomial-rat-haus-div} are equalities in general position, but in some cases are strict inequalities as the following example shows.

\para

\begin{example}\label{ex-quintic-dim-degenado}
Let $D=D_H+D_S$ be a $5$-rational Hausdorff divisor, where $D_S=\sum_{i=1}^{6} 2P_i$ with $P_i=(i^3:i^2:1)$. Theorem \ref{theorem-dim-rational-hauss-divisor} ensures that $\dim({\cal H}(5,D))\geq 9-12=-3$; i.e. in general ${\cal H}(5,D)=\emptyset$. However, taking
$D_H= \sum (\alpha:1:0)$, with $p(\alpha)=0$, where ($\mu_4,\mu_3\in \C$)
$$ p(t)={t}^{5}+\mu_{{4}}{t}^{4}+\mu_{{3}}{t}^{3}-{\frac {820955079}{
2000}}\mu_{{4}}{t}^{2}-{\frac {39925319}{4000}}\mu_{{3}}{t
}^{2}-{\frac {37424614507}{4000}}{t}^{2}+{\frac {1609223259}{200}}
\mu_{{4}}t$$ $$+{\frac {73371752447}{400}}t+{\frac {78188299}{400}}
\mu_{{3}}t-{\frac {1369716346817}{1000}}-{\frac {30038293649}{500}
}\mu_{{4}}-{\frac {1459046589}{1000}}\mu_{{3}},$$
it holds that $\dim({\cal H}(5,D))=0$. Indeed, in this case, ${\cal H}(5,D)$ consists in the curve defined by

\para

\noindent $-59099417781138{y}^{4}z\mu_{{4}}-40586935102980{x}^{3}{z}^{2}
\mu_{{4}}-66732447304488{x}^{2}{z}^{3}\mu_{{4}}-
1111124672000{y}^{2}{z}^{3}\mu_{{3}}+1295029858022677{x}^{2}y{
z}^{2}+872619620320000xy{z}^{3}+1948148631{x}^{4}z\mu_{{3}}-
1956340631{x}^{2}{y}^{3}\mu_{{3}}+196000\mu_{{4}}{x}^{4}y+
935066250218170x{y}^{3}z-11652560640000y{z}^{4}\mu_{{4}}-
1437222765209{y}^{4}z\mu_{{3}}+80264569742{x}^{4}z\mu_{{4}
}+796735296000x{z}^{4}\mu_{{3}}+32965011072000x{z}^{4}\mu_
{{4}}-985374741890{x}^{3}{z}^{2}\mu_{{3}}+196000\mu_{{3}}{
x}^{3}{y}^{2}+1577038793820x{y}^{4}\mu_{{4}}-279861120000y{z}^
{4}\mu_{{3}}+44812924280488{y}^{3}{z}^{2}\mu_{{4}}-
45598986048000{y}^{2}{z}^{3}\mu_{{4}}+38312266510x{y}^{4}
\mu_{{3}}+1081197980884{y}^{3}{z}^{2}\mu_{{3}}+
189114415140000x{y}^{2}{z}^{2}+267281628256132{x}^{2}{y}^{2}z-
80453597742{x}^{2}{y}^{3}\mu_{{4}}-1628676180884{x}^{2}{z}^{3}
\mu_{{3}}-35862859923030{x}^{3}yz+20736000000{z}^{5}\mu_{{
3}}+870912000000{z}^{5}\mu_{{4}}-285973131444{y}^{5}\mu_{{
3}}-11775011110408{y}^{5}\mu_{{4}}-268037400960000y{z}^{4}-
1036755669824000{y}^{2}{z}^{3}+1026724988552452{y}^{3}{z}^{2}-
1346097219374677{y}^{4}z+754414452288000x{z}^{4}+35952158699030x
{y}^{4}-1515620122272452{x}^{2}{z}^{3}-1833806110843{x}^{2}{y}^{3}
-925538635338170{x}^{3}{z}^{2}+1830618086843{x}^{4}z+
20176128000000{z}^{5}-268464403976132{y}^{5}+196000{x}^{5}+
953164576000xy{z}^{3}\mu_{{3}}+38702489280000xy{z}^{3}\mu_
{{4}}+218934548000x{y}^{2}{z}^{2}\mu_{{3}}+8600406696000x{y}^{
2}{z}^{2}\mu_{{4}}+997711569890x{y}^{3}z\mu_{{3}}+
41040600622980x{y}^{3}z\mu_{{4}}+1375331229209{x}^{2}y{z}^{2}
\mu_{{3}}+56732067257138{x}^{2}y{z}^{2}\mu_{{4}}+
284276051444{x}^{2}{y}^{2}z\mu_{{3}}+11716423378408{x}^{2}{y}^
{2}z\mu_{{4}}-38158910510{x}^{3}yz\mu_{{3}}-1572326561820{
x}^{3}yz\mu_{{4}}. $ \hfill $\Box$
\end{example}

\section{Parametrization of Rational Hausdorff Linear Systems}\label{sec-param-rat-Haus-Spaces}

Let $\K$ be a subfield of $\C$, and let $D=D_H+D_S$ be a  $\K$-definable $n$-rational Hausdorff divisor. Our goal in this section is to parametrize the curve $\curva({\cal H}(n,D))$ associated to $D$; that is the curve defined, over the algebraic closure of $\C(\Lambda)$, by the defining polynomial $H(\Lambda,x,y,z)$ of  the rational Hausdorff linear space ${\cal H}(n,D)$ (see Def. \ref{def-curve-asoc-linear-sytema}). Recall that, by Theorem \ref{theorem-main-property-genus-divisors}, if $D$ is irreducible, then $\curva({\cal H}(n,D))$ is rational.

\para

Thus, throughout this section we assume that $D$ is irreducible (see Def. \ref{Def-irred} and Theorem \ref{theorem-irred-2}) which, in particular, implies that  ${\cal H}(n,D)$ is not empty (see also Theorem \ref{theorem-dim-rational-hauss-divisor} and Corollary \ref{col-dim-monomial-rat-haus-div}). Moreover, let $H(\Lambda,x,y,z)\in \K[\Lambda][x,y,z]$, where $\Lambda$ is a set of parameters, be the defining polynomial of ${\cal H}(n,D)$; observe that the $\K$-definability of $D$ implies that $H$ is a polynomial over $\K$.

\para

  But before going into details, let us recall, at least intuitively, how the parametrization algorithms, based on adjoint curves, work. Since, we will be dealing only with ordinary singularities we simplify the exposition to that case; for further details, see \cite{SWP}. Say that $\cc$ is a rational projective curve of degree $k$. The simplest case is when $\cc$ is monomial; let $P$ be the $(k-1)$ fold-point of $\cc$. In this situation, the intersection of $\cc$ with $\curva({\cal H}(1,P))$ consists in $P$ and an additional  point that depends rationally on a parameter. This last point is indeed a parametrization of $\cc$. This method is called {\sf parametrization by lines}. In general, let $\{P_1,\ldots,P_s\}$ be the singularities of $\cc$, then  an {\sf adjoint curve to $\cc$} of degree $\ell$ (in general, $\ell\geq k-2$) is any curve in the linear system of curves
\[ {\cal H}(\ell,\sum_{i=1}^{s} (\mult(\cc,P_i)-1)P_i ).\]
Let ${\cal A}_{\ell}(\cc)$ denote the linear system above, that is the linear system of all adjoints to $\cc$ of degree $\ell$. Because of the genus formula and the dimension of ${\cal A}_{\ell}(\cc)$ it holds that taking a finite set of simple points $\{Q_1,\ldots,Q_r\}$ of $\cc$, for a suitable  $r$, and considering ${\cal H}^*:={\cal A}_{\ell}(\cc)\cap {\cal H}(\ell, Q_1+\cdots+Q_r)$ it holds that the intersection of $\cc$ with $\curva({\cal H}^*)$ consists in $\{P_1,\ldots,P_s\}\cup \{Q_1,\ldots,Q_r\}$ and an additional  point that depends rationally on a parameter. This last point is indeed a parametrization of $\cc$. Let us assume that the homogeneous form defining  $\cc$ has coefficients in  $\K$. Then
an important property, of these type of algorithms, is that the coefficients of the parametrization (field of parametrization) are in $\K$ (if $\cc$ was parametrized by lines) or in the smallest field containing $\K$ and the coefficients of the chosen points $\{Q_1,\ldots,Q_s\}$.

\para

As we said, our goal is to parametrize $\curva({\cal H}(n,D))$, but sometimes, we will also parametrize  the curve $\curva(\overline{\cal H})$ associated to a non-empty  linear subsystem $\overline{\cal H}$ of ${\cal H}(n,D)$. Applying the well-known parametrization algorithms, since the coefficients of the input curve are in $\K(\Lambda)$, one derives a rational parametrization of $\curva({\cal H}(n,D))$ over the algebraic closure of $\C(\Lambda)$. The challenge is to parametrize $\curva({\cal H}(n,D))$ over the smallest possible field extension of $\K(\Lambda)$. We start observing that, as a consequence of Hilbert-Hurwitz's Theorem (see Theorem 5.8. in \cite{SWP}) and Tsen's Theorem (Corollary 4 in \cite{Shafa}, Vol. I. pg. 73), every irreducible linear subsystem of dimension 0 or 1 of ${\cal H}(n,D)$ is parametrizable over $\C(\overline{\Lambda})$, where $\overline{\Lambda}$ are the parameters involved in the definition of the subsystem. Nevertheless, as a consequence of the Hausdorff divisor, we can improve this statement (note that no hypothesis on the dimension is required). We recall that {\sf proper} means that the parametrization defines a 1:1 map from a non-empty Zariski open subset of the parameter space and the curve.

\para

\begin{theorem}\label{Theorem-main-param} {\sf [General Parametrization Theorem]}
There exists a rational proper parametrization of $\curva({\cal H}(n,D))$ with coefficients in  $\L(\Lambda)$, where $\L$ is a finite algebraic extension of $\K$ of degree at most $n$. Furthermore, the degree of the extension is the lowest degree of the nontrivial irreducible factors, in $\K[\Lambda][x,y]$ of $H(\Lambda,x,y,0)$.
\end{theorem}
\begin{proof} Since $D_H$ is Hausdorff, and $\deg(D_H)=\deg(\curva({\cal H}(n,D)))$, by B\'ezout's Theorem it holds that all points of $\curva({\cal H}(n,D))$ on the line $z=0$ are simple. Moreover, these points are over $\C$. Furthermore, since $D_H$ is $\K$-definable, these points at infinity form a $\K$-conjugate family of points that can be decomposed as union of families, each defined by a factor of $H(\Lambda,x,y,0)$ in $\K[\Lambda][x,y]$; say that $k$ is the lowest degree of these factors. On the other hand, since $D_S$ is $\K$-definable, one has that the linear system of $n$-degree adjoint curves to $\curva({\cal H}(n,D))$ can be defined over $\K$ (see Theorem 4.66. in \cite{SWP}). Therefore,
 using the parametrization algorithm by $n$-degree adjoint curves (see Section 4.8 in \cite{SWP}) and taking the simple point in one of the families of cardinality $k$,   one deduces that $\curva({\cal H}(n,D))$ can be properly parametrized over $\L(\Lambda)$, where $\L$ is a finite algebraic extension of $\K$ of degree $k$.
%
\end{proof}

\para

From the previous proof one can derive an algorithm to parametrize $\curva({\cal H}(n,D))$ over $\L(\Lambda)$. Indeed, the extension $\L$ is the extension needed to express the simple points in $D_H$ used in the parametrization algorithm. In the following we analyze how to decrease the degree of the extension in some special cases.

\para

\begin{corollary}\label{corol-rat-point-at-infinity}
If one of the points in $D_H$ has coordinates over $\K$, there exists a rational proper parametrization of $\curva({\cal H}(n,D))$ with coefficients in  $\K(\Lambda)$.
\end{corollary}

\para

\begin{example}\label{example-cuadrics-param}
Let $D$ be the 4-rational divisor in Example \ref{example-cuadrics-2}.
Since $D_H$ has points in $\Q$, Corollary \ref{corol-rat-point-at-infinity} ensures that $\curva({\cal H}(4,D))$ can be parametrized over $\Q(\Lambda)$. Indeed, if we take $2$-degree adjoints and we use the simple point $(1:0:0)$, we get the parametrization
\[ \left(\frac{A_1(t)}{A_2(t)},\frac{A_3(t)}{A_4(t)},1\right) \]
where

\para

\noindent $A_1(t)=-238\lambda_1{t}^{3}\lambda_2+2240\lambda_2{t}^{2}\lambda_1+98{
\lambda_1}^{2}{t}^{3}-2787{\lambda_2}^{2}{t}^{2}+1470{\lambda_1}^{
2}-6986\lambda_1\lambda_2-539{\lambda_1}^{2}t+8328{\lambda_2}^{2
}+1792\lambda_1\lambda_2t-441{\lambda_1}^{2}{t}^{2}-1209{\lambda_2}^{2}t,$ \\ \\
$A_2(t)= 14\lambda_2 \left( 2\lambda_2-2\lambda_2{t}^{2}-17a_{
{3,1}}{t}^{3}+17\lambda_2t-7\lambda_1t+7\lambda_1{t}^{3}
 \right),$ \\ \\
$A_3(t)= 486\lambda_2+77\lambda_1t+34\lambda_2{t}^{3}+63a_{{0
,3}}{t}^{2}-145\lambda_2t-147\lambda_2{t}^{2}-14\lambda_1{t}^{3}
-210\lambda_1, $\\ \\
$A_4(t)= 14\lambda_2 \left( -1+{t}^{2} \right)$.
\\ \\
\noindent Here $\Lambda=(\lambda_1,\lambda_2)$. \hfill $\Box$
\end{example}

\para

\begin{example}\label{example-cuadrics-param-2}
Let us consider the 4-degree $\Q$-definable rational Hausdorff divisor
\[ D=\sum_{t^4-4=0} (1:t:0)+2\sum_{t^3+1=0} (t:t^2:1). \]
Note that
\[ D_H=\sum_{t^4-4=0} (1:t:0),D_S=2\sum_{t^3+1=0} (t:t^2:1). \]
The rational Hausdorff linear system ${\cal H}(4,D)$ associated to $D$ is given by the polynomial (where $\Lambda=(\lambda_1,\lambda_2)$)

\para

\noindent $H(\Lambda,x,y,z)=-\lambda_1{z}^{4}+\lambda_{2}{y}^{2}{z}^{2}-\lambda_1{y}^{3}z-4a_{{4,0
}}{y}^{4}-3\lambda_1xy{z}^{2}-8\lambda_{2}x{y}^{2}z-4\lambda_{2}{x}^
{2}{z}^{2}-2\lambda_{2}{x}^{2}yz+\lambda_1{x}^{3}z+\lambda_{2}{x}^{4}$.

\para

\noindent Observe that $\dim({\cal H}(4,D))=1$ and compare to Theorem \ref{theorem-dim-rational-hauss-divisor}. We observe that the Hausdorff divisor can be expressed by conjugate families as
\[ D_H=\sum_{t^2-2=0} (1:t:0) +\sum_{t^2+2=0} (1:t:0). \]
Corollary \ref{corol-rat-point-at-infinity} ensures that $\curva({\cal H}(4,D))$ can be parametrized over $\Q(\sqrt{2})(\Lambda)$. Indeed, if we take $2$-degree adjoints and we use the simple point $(1:\sqrt{2}:0)$, we get the parametrization
\[ \left(\frac{A_1(t)}{B(t)},\frac{A_2(t)}{B(t)},1\right) \]
where

\para

\noindent $A_1(t)=\frac{1}{14} (-4+\sqrt {2} )  ( \sqrt {2}{t}^{4}{\lambda_1}
^{2}+33{\lambda_1}^{2}t\sqrt {2}+16t\lambda_{2}\lambda_1-4{\lambda_{2}}^{2}{t}^{3}\sqrt {2}+12\lambda_{2}{t}^{3}\lambda_1+16{t}^{3}
\sqrt {2}{\lambda_1}^{2}+16{t}^{4}\lambda_1\lambda_{2}+12{\lambda_1}
^{2}{t}^{2}\sqrt {2}+32{t}^{2}\lambda_{2}\lambda_1+20{\lambda_1}^{2}
t+14{t}^{4}{\lambda_{2}}^{2}-16{t}^{3}{\lambda_{2}}^{2}+8{t}^{3}{\lambda_1}^{2}+48{\lambda_1}^{2}{t}^{2}+4{t}^{4}{\lambda_1}^{2}+4
\sqrt {2}{t}^{4}\lambda_1\lambda_{2}+4\lambda_{2}\lambda_1t\sqrt {2}+31
{t}^{3}\lambda_{2}\lambda_1\sqrt {2}+8\lambda_{2}{t}^{2}\sqrt {2}\lambda_1+18{\lambda_1}^{2}+8\sqrt {2}{\lambda_1}^{2}),$ \\ \\
$A_2(t)=-\frac{1}{7} ( 1+2\sqrt {2}) ( -16{\lambda_{2}}^{2}{t}^{3}\sqrt {2}+3{\lambda_1}^{2}{t}^{2}\sqrt {2}+8t\lambda_{2}\lambda_1-
12{\lambda_1}^{2}t\sqrt {2}+10\lambda_{2}{t}^{3}\lambda_1-4{t}^{4}
\lambda_1\lambda_{2}-50{t}^{2}\lambda_{2}\lambda_1+6{\lambda_1}^{2}t-7
{t}^{4}{\lambda_{2}}^{2}+8{t}^{3}{\lambda_{2}}^{2}-2{t}^{3}\sqrt {2}
{\lambda_1}^{2}-12{\lambda_1}^{2}{t}^{2}+\sqrt {2}{t}^{4}\lambda_1 \lambda_{2}-16\lambda_{2}\lambda_1t\sqrt {2}+{t}^{3}{\lambda_1}^{2}+16\lambda_{2}{t}^{2}\sqrt {2}\lambda_1-7{\lambda_1}^{2}-20{t}^{3}\lambda_2\lambda_1\sqrt {2}),$ \\ \\
$B(t)=\lambda_{2}t \left( {t}^{3}\lambda_1+12\lambda_1t+4\lambda_{2}{t}^{3}+
16t\lambda_{2}+4{t}^{2}\lambda_1\sqrt {2}+8\lambda_1\sqrt {2}+8\lambda_{2}{t}^{2}\sqrt {2} \right).
  \hfill \Box$

\end{example}

\para

\begin{corollary}\label{corol-lower-dim}
If $\dim({\cal H}(n,D))>0$,  for every $P\in \mathbb{P}^2(\K)$ such that $\overline{\cal H}:={\cal H}(n,D)\cap {\cal H}(n,P)$ is irreducible, then
 $\curva(\overline{\cal H})$ can be rationally and properly parametrized over $\K(\Lambda)$.
\end{corollary}
\begin{proof}
It follows by using $P$ in the parametrization algorithm.
\end{proof}

\para

\begin{example}
Let $D$ be as in Example \ref{example-cuadrics-param-2}. Since $\dim({\cal H}(4,D))=1$ we apply Corollary \ref{corol-lower-dim}. We take a point $P:=(a:b:1)\in {\mathbb P}^2(\C)$ and we consider $\overline{\cal H}={\cal H}(4,D)\cap {\cal H}(4,P)$. In order to avoid reducibility, computations show that $P$ has to be taken not satisfying the equation
\[ (a-1-b)(a^2+a+ab-b+1+b^2)=0. \]
The defining polynomial of $\overline{\cal H}$ is

\para

\noindent $\overline{H}(x,y,z)=-24\,xy{z}^{2}a{b}^{2}-6\,xy{z}^{2}{a}^{2}b+24\,x{y}^{2}zab+6\,{x}^{2}
yzab+4\,{y}^{4}-{x}^{4}-12\,xy{z}^{2}{a}^{2}+3\,xy{z}^{2}{a}^{4}+8\,x{
y}^{2}z{b}^{3}-8\,x{y}^{2}z{a}^{3}+12\,{x}^{2}{z}^{2}ab+2\,{x}^{2}yz{b
}^{3}-2\,{x}^{2}yz{a}^{3}+8\,{x}^{3}za{b}^{2}+2\,{x}^{3}z{a}^{2}b-3\,{
y}^{2}{z}^{2}ab-8\,{y}^{3}za{b}^{2}-2\,{y}^{3}z{a}^{2}b+3\,xy{z}^{2}{b
}^{2}-12\,xy{z}^{2}{b}^{4}-4\,{y}^{3}z{b}^{4}-{x}^{3}z{a}^{4}+4\,{x}^{
2}{z}^{2}{b}^{3}+{y}^{3}z{a}^{4}-{y}^{2}{z}^{2}{b}^{3}-4\,{x}^{2}{z}^{
2}{a}^{3}+12\,{y}^{4}ab-{x}^{3}z{b}^{2}-2\,{z}^{4}{a}^{2}b+4\,{x}^{3}z
{b}^{4}+{y}^{3}z{b}^{2}+{y}^{2}{z}^{2}{a}^{3}+8\,x{y}^{2}z-8\,{z}^{4}a
{b}^{2}-4\,{y}^{3}z{a}^{2}-3\,{x}^{4}ab+2\,{x}^{2}yz+4\,{x}^{3}z{a}^{2
}-{y}^{2}{z}^{2}+4\,{x}^{2}{z}^{2}+{z}^{4}{b}^{2}-4\,{z}^{4}{b}^{4}-4
\,{z}^{4}{a}^{2}+{z}^{4}{a}^{4}+4\,{y}^{4}{b}^{3}-4\,{y}^{4}{a}^{3}-{x
}^{4}{b}^{3}+{x}^{4}{a}^{3}. $

\para

In this situation, we consider the system of conics ${\cal H}^*={\cal H}(2,\sum_{t^3+1=0} (t:t^2:1))\cap {\cal H}(2,P)$, that is defined by

\para

\noindent $H^*(t,x,y,z)={z}^{2}{b}^{2}+{z}^{2}a-tyz{b}^{2}-tyza-{y}^{2}+{y}^{2}tb-{y}^{2}ab-{y
}^{2}t{a}^{2}-xz+xztb-xzab-xzt{a}^{2}+xy{b}^{2}+xya+t{x}^{2}{b}^{2}+t{
x}^{2}a.$

\para

Then, the intersection of $\overline{\cal H}$ and ${\cal H}^*$ provides a parametrization of $\curva(\overline{\cal H})$ with coefficients in $\Q(a,b)$; we do not show the output here because it is too large. Instead, we illustrate it with particular values of $P$, for instance $P=(1:1:1)$. In this case, we get the parametrization
\[ \left({\frac {1024+1024 t+960 {t}^{4}+4096 {t}^{2}+4352 {t}^{3}}{
-4 \left( 16-16 t-4 {t}^{2} \right)  \left( 16+16 t+12 {t}^{2}
 \right) }}, {\frac {1024+2048 t-192 {t}^{4}+3584 {t}^{2}+512
{t}^{3}}{ 4\left( 16-16 t-4 {t}^{2} \right)  \left( 16+16 t+12 {t}^
{2} \right) }}, 1 \right).
\]
 Similarly, for $P=(0:1:1)$ we get
\[ \left(-\frac{1}{2}\,{\frac {18-15\,t+6\,{t}^{2}+74\,{t}^{3}+21\,{t}^{4}}{ \left( -{
t}^{2}-6\,t+9 \right)  \left( 3\,{t}^{2}+2\,t+1 \right) }},\frac{1}{2}\,{
\frac {-3\,t-6\,{t}^{4}+75\,{t}^{2}+23\,{t}^{3}}{ \left( -{t}^{2}-6\,t
+9 \right)  \left( 3\,{t}^{2}+2\,t+1 \right) }}, 1 \right).
\]
Both parametrizations have coefficients in $\Q$.
\hfill $\Box$
\end{example}

\para

Let us assume that $D$ is monomial, then one can parametrize $\curva({\cal H}(n,D))$ by lines (see Section 4.6 in \cite{SWP}). In addition, since $D$ is $\K$ definable, then the field of  parametrization is $\K(\Lambda)$. Therefore, one has the following theorem.

\para

\begin{theorem}\label{theorem-param-monomial}
Let $D$ be monomial, then there exists a rational proper parametrization of $\curva({\cal H}(n,D))$ with coefficients in  $\K(\Lambda)$.
\end{theorem}

\para

\begin{example}\label{example-param-monomial}
We consider the divisor
\[ D= \sum_{t^4+1=0} (t:1:0) +3(0:0:1). \]
$D$ is a $\Q$-definable $4$-monomial Hausdorff divisor. So, by Corollary \ref{col-dim-monomial-rat-haus-div}, $\dim({\cal H}(4,D))=4$. Indeed, the defining polynomial of ${\cal H}(4,D)$ is
\[ H=\lambda_{1}{y}^{3}z+\lambda_{2}{y}^{4}+\lambda_{3}x{y}^{2}z+a_{{2,1}}{x}^{2}
yz+\lambda_{4}{x}^{3}z+\lambda_{2}{x}^{4}.
\]
We observe that $H$ is irreducible over $\C$, and hence $D$ is irreducible (see Theorem \ref{theorem-irred-2}).
Now, parametrizing with the pencil of lines $ty+x=0$ one gets the parametrization of the linear system
\[ \cP(t)=\left(\frac{-t \left( t\lambda_{3}-\lambda_{1}+{t}^{3}\lambda_{4}-{t}^{2}a_{{2,1}}
 \right)}{\lambda_{2} \left( {t}^{4}+1 \right)} ,\frac{t\lambda_{3}-\lambda_{1}+{t}^{3}\lambda_{4}-{t}^{2}a_{{2,1}}}{\lambda_{2} \left( {t}^{4}+1 \right)},1 \right). \]
\end{example}

\para



\begin{theorem}\label{theorem-param-rat-sing-point}
Let $D_S$ have at least a triple point over $\K$, then there exists a rational proper parametrization of $\curva({\cal H}(n,D))$ with coefficients in  $\K(\Lambda)$.
\end{theorem}
\begin{proof} Using the triple point one can generate families of $(n-3)$ conjugate points over $\K(\Lambda)$ (see Section 3.3 in \cite{SWP}) to afterwards parametrize with $(n-2)$-degree adjoint curves (see Section 4.7 in \cite{SWP}).
\end{proof}

\section{Application to the Approximate Parametrization Problem}

Given a non-rational irreducible curve, the approximate parametrization problem consists in providing a rational curve being at close Hausdorff distance of the input curve; see the Introduction for further details.
In this section we show, as a sample of application of the ideas developed, that every Hausdorff curve (see definition below) can always be parametrized approximately.

\para

\begin{definition}\label{def-curva-hauss}
We say that an affine  plane algebraic curve $\cc$ is a {\sf Hausdorff curve} if $\card(\cc^{\infty})=\deg(\cc)$; recall that $\cc^{\infty}$ denotes the points at infinity of $\cc$.
\end{definition}

\para

\begin{example}\label{ex-curva-haus}
Observe that all lines are Hausdorff and the only conics that are not Hausdorff are the parabolas. For degree 3 or higher the number of possibilities increases.
\end{example}

\para

\begin{remark}\label{remark-haus-curve}
Observe that, if $\cc$ is a $\K$-definable Hausdorff curve of degree $n$, then
\[ D=\sum_{P\in \cc^{\infty}} P \]
is an $n$-degree $\K$-definable Hausdorff divisor. We call $D$ the {\sf Hausdorff divisor associated to $\cc$}.
\end{remark}

\para

The following theorem states the main property of Hausdorff curves.

\para

\begin{theorem}\label{theorem-haus-curve}
Let  $\cc$ be a real irreducible affine Hausdorff curve of degree $n$, and let $D=\sum_{i=1}^{n}  (a_i:b_i:0)$ be its associated Hausdorff divisor. Then, for every point $P=(a:b:1)\in {\mathbb P}^2(\C)$, such that $ab_i-ba_i\neq 0$ with $i=1,\ldots,n$, $D+(n-1)P$ is an irreducible Hausdorff monomial divisor.
\end{theorem}
\begin{proof} Let $\overline{D}=D+(n-1)P$. Taking $\overline{D}_H=D$ and $\overline{D}_S=(n-1)P$, one has that $\overline{D}$ is Hausdorff and monomial. In order to prove that $\overline{D}$ is irreducible, by Corollary \ref{corollary-irred}, we prove that the defining polynomial $\overline{H}(\Lambda,x,y,z)$ of ${\cal H}(n,\overline{D})$ is irreducible over $\C(\Lambda)$. Moreover, since $P$ has coefficients in $\C$, we can consider w.l.o.g. that $P=(0:0:1)$; otherwise one performs a suitable linear change over $\C$. In this situation, the proof is analogous to the proof of Lemma \ref{lemma-irred}.
\end{proof}

\para

The next result follows from Theorem \ref{theorem-haus-curve}, Theorem \ref{theorem-param-monomial}, Theorem \ref{theorem-haus-divisor}, and Corollary \ref{col-dim-monomial-rat-haus-div}.

\para

\begin{corollary}\label{corollary-haus-curve-1}
Let $\K$ be a subfield of $\C$, let $\cc$ be a real irreducible affine $\K$-definable Hausdorff curve of degree $n$. Then, there exist infinitely many real monomial plane curves $\dd$, parametrizable over $\K$, such that $\ddd(\cc\cap \R^2,\dd\cap \R^2)<\infty$. Furthermore, for any fixed point $P$, chosen  as in Theorem \ref{theorem-haus-curve}, the dimension of the linear system of $n$-degree monomial curves, having $P$ as singular point, is $n$. That is, for every $P$ chosen as above, there exists an $n$-dimensional linear system where all irreducible curves are solutions of the approximate parametrization problem applied to $\cc$.
\end{corollary}

\para

From the previous result, one may proceed as follows. Let us say that we are given a real irreducible affine Hausdorff curve $\cc$ of degree $n$, and we want to provide a rational curve $\dd$, at finite Hausdorff distance of $\cc$, passing through a fixed affine point $P$. We may assume that $P$ satisfies the conditions in Theorem \ref{theorem-haus-curve}, otherwise we apply an small perturbation to $P$. In this situation, one computes
$\overline{\cal H}={\cal H}(n, D+(n-1)P)$, where $D$ is the Hausdorff divisor associated to $\cc$. We know that almost all curves in $\overline{\cal H}$ are irreducible, and hence rational. Moreover, we know that $\dim(\overline{\cal H})=n$. That is, we still have $n$ degrees of freedom to chose a suitable (under the requirements stated by the user) rational curve to our particular problem. For instance, one may look for a rational curve in $\overline{\cal H}$ under the criterium of minimizing the Hausdorff distance, or reducing the length of the coefficients in the parametrization, or passing thought the ramification points of $\cc$, or having particular tangents at particular points, etc.

\para

We finish this section, illustrating these ideas by an example

\para

\begin{example}\label{ex-param-approx}
We consider the affine curve $\cc$ defined by
\[4+2 y-5 {y}^{2}-9 {y}^{3}+6 {y}^{4}+x-7 xy-5 x{y}^{2}-6 {x}^{2}
+6 {x}^{2}y-3 {x}^{3}-6 {x}^{4}.\]
$\cc$ is real, irreducible and  has degree 4. Moreover, $\cc^{\infty}=\{ (1:\pm 1:0),(1:\pm \ii:0)\}$. Therefore, $\cc$ is Hausdorff and its associated Hausdorff divisor is
\[ D=\sum_{t^4=1}(1:t:0). \]
On the other hand, $\cc$ has genus 3; i.e. it is smooth.  We take a point $P=(41/64: -1/32: 1)$ satisfying the conditions of Theorem \ref{theorem-haus-curve}; $P$ has been taken as an approximation of a ramification points of $\cc$. Then, $\overline{D}=D+3P$ is an irreducible Hausdorff monomial divisor. The associated linear system ${\cal H}(4,\overline{D})$ is defined by (here $\Lambda=(\lambda_1,\ldots,\lambda_5)$)

\para

\noindent $H(\Lambda,x,y,z)=\lambda_{{4}}{x}^{3}z+{\frac {2825745}{524288}} y{z}^{3}\lambda_{{4}}
+1024 x{y}^{2}z\lambda_{{2}}+1312 x{y}^{2}z\lambda_{{3}}+{\frac {
1681}{4}} {y}^{2}{z}^{2}\lambda_{{3}}-\lambda_{{4}}{y}^{4}+13448 {y}
^{3}z\lambda_{{3}}+{\frac {7236657}{512}} {y}^{3}z\lambda_{{4}}+
\lambda_{{4}}{x}^{4}+32 {x}^{2}yz\lambda_{{3}}-{\frac {68921}{2048}}
 xy{z}^{2}\lambda_{{4}}+\lambda_{{3}}{x}^{2}{z}^{2}+1312 {y}^{2}{z}^
{2}\lambda_{{2}}+\lambda_{{2}}x{z}^{3}+41 xy{z}^{2}\lambda_{{3}}+
32768 {y}^{3}z\lambda_{{1}}+96 y{z}^{3}\lambda_{{1}}+3072 {y}^{2}{z
}^{2}\lambda_{{1}}+64 xy{z}^{2}\lambda_{{2}}+{\frac {8979}{64}} {x}^
{2}yz\lambda_{{4}}+\lambda_{{1}}{z}^{4}-{\frac {2825809}{16384}} {y}^
{2}{z}^{2}\lambda_{{4}}+{\frac {41}{2}} y{z}^{3}\lambda_{{2}}+{\frac
{149609}{64}} x{y}^{2}z\lambda_{{4}}+20992 {y}^{3}z\lambda_{{2}}.$

\para

As expected, $\dim({\cal H}(4,\overline{D})=4$. Moreover, for every $\Lambda_0\in \C^5$, such that $H(\Lambda_0,x,y,z)$ is irreducible over $\C$, we
get a monomial curve. Furthermore, the affine curve $H(\Lambda_0,x,y,1)$ is monomial and is at finite distance of $\cc$. Since we have 4 degrees of freedom, we choose the curve such that it passes through 4  points of $\cc$. We intersect $\cc$ with the lines $y=\pm 3$ to get
\[  Q_1=\left({\frac {89}{32}}:-3:1\right), Q_2=\left(-{\frac {101}{32}}:-3:1\right),Q_3=\left({\frac {65}{32}}:3:1\right),Q_4=\left(-{\frac {103}{32}}:3:1\right).
\]
We consider $\overline{\overline{\cal H}}={\cal H}(4,D+3P+Q_1+Q_2+Q_3+Q_4)$. We note that $\dim(\overline{\overline{\cal H}})=0$ and consists in the curve

\para

\noindent
\begin{flushleft}
$G(x,y,z)=-11189780504385617373808 y{z}^{3}-64177446384507906894080 {y}^{2}{z}
^{2}+25328929045126690271232 {y}^{3}z-68315663351181964574720 {x}^{3
}z+69446473202369720695808 {x}^{2}{z}^{2}-30949472647714110913696 x{
z}^{3}-24897211394328530780160 {y}^{4}+24897211394328530780160 {x}^{
4}+28677478743593794827264 xy{z}^{2}+104113819442735106875392 x{y}^{
2}z-17303699534378810261504 {x}^{2}yz+5094649843686955824985 {z}^{4}, $
\end{flushleft}

\begin{figure}[h]
  \centering
\includegraphics[width=4.5cm]{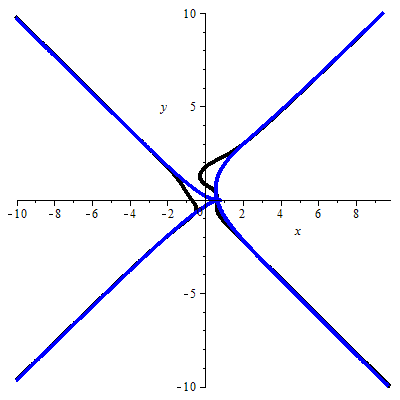} \includegraphics[width=4.5cm]{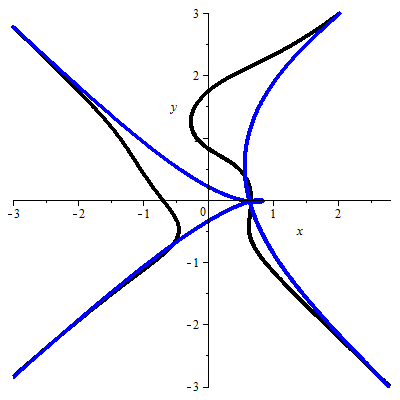}\includegraphics[width=4.5cm]{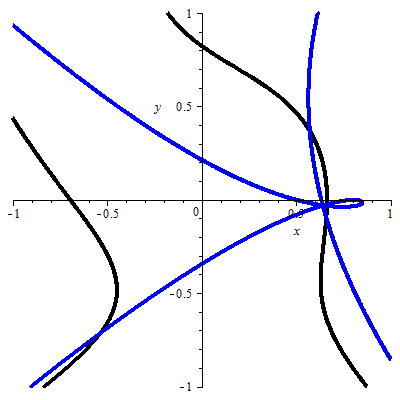}
\caption{Plot of $\cc$ and $\dd$ in Example \ref{ex-param-approx}. {\sf Left:} general view. {\sf Center:} Non-asymptotic area. {\sf Left:} Singular area.}
\label{figure-param}
\end{figure}

\para

\noindent that is rational and can be parametrized by lines through $P$ as
\[ \left(\frac {1}{94975324227632640}\frac{A_1(t)}{B(t)},- \frac{1}{47487662113816320}\frac{A_2(t)}{B(t)},1\right) \]
where

\para

\noindent \begin{flushleft}  $A_1(t)=208880643591165188824385+1309845452973236446822400t^4+3152348304551138336556032t^2+1326609920992631925943776t
+3321871574175160774459392t^3$,
\end{flushleft}

\para

\noindent \begin{flushleft}  $A_2(t)=30949472647714110913696t+68315663351181964574720t^3+69446473202369720695808t^2+5094649843686955824985+24897211394328530780160t^4$
\end{flushleft}

\para

\noindent $B(t)=2825745+41312256t^2+42991616t^3+16777216t^4+17643776t$.

\para

Furthermore, the affine curve defined $\dd$ by $G(x,y,1)$ is at finite distance of $\cc$ (see Figure \ref{figure-param}). \hfill$\Box$.

\end{example}

\para

\noindent {\bf Acknowledgments:} This work was developed, and partially supported, under the research project
MTM2011-25816-C02-01. All authors belongs to  the Research Group ASYNACS (Ref. CCEE2011/R34).

\end{document}